\documentclass[11pt]{amsart}

\usepackage{amsfonts,amssymb,amsthm,eucal,color,bbm,verbatim,graphicx,hyperref,enumitem}
\usepackage[french,english]{babel}
\usepackage{frcursive}

\usepackage[T1]{fontenc}
\usepackage{aurical}

\usepackage[margin=1.25in]{geometry}
\calclayout

\newtheorem{thm}{Theorem}

\newtheorem{lemma}[thm]{Lemma}

\theoremstyle{definition}

\newcommand{\R}{\mathbb{R}}
\newcommand{\F}{\mathbb{F}}
\newcommand{\E}{\mathbb{E}}

\newcommand{\C}{\mathbb{C}}

\DeclareMathOperator{\tr}{tr}

\DeclareMathOperator{\Cov}{Cov}
\newcommand{\inprod}[2]{\left\langle #1, #2 \right\rangle}

\newcommand{\abs}[1]{\left\vert #1 \right\vert}
\newcommand{\norm}[1]{\left\Vert #1 \right\Vert}

\DeclareMathOperator{\var}{Var}

\newcommand{\Set}[2]{\left\{#1 \mathrel{} \middle| \mathrel{} #2
  \right\}}

\DeclareMathOperator{\diag}{diag}

\newcommand{\Lgen}{\mathrm{L}}

\newcommand{\Mat}{M}

\newcommand{\ind}[1]{\mathbbm{1}_{#1}}

\allowdisplaybreaks[3]

\author{David Grzybowski}

\address{Department of Mathematics, Applied Mathematics, and
  Statistics, Case Western Reserve University, 10900 Euclid Ave.,
  Cleveland, Ohio 44106, U.S.A.}

\email{drgrizy67@gmail.com}

\author{Mark Meckes}

\address{Department of Mathematics, Applied Mathematics, and
  Statistics, Case Western Reserve University, 10900 Euclid Ave.,
  Cleveland, Ohio 44106, U.S.A.}

\email{mark.meckes@case.edu}

\title[Stein's method, Markov processes, and random matrices]{Stein's
  method, Markov processes, and linear eigenvalue statistics of random
  matrices}

\begin{document}

\maketitle

\begin{abstract}
  We show how the infinitesimal exchangeable pairs approach to Stein's
  method combines naturally with the theory of Markov semigroups.  We
  present a multivariate normal approximation theorem for functions of
  a random variable invariant with respect to a Markov semigroup.
  This theorem provides a Wasserstein distance bound in terms of
  quantities related to the infinitesimal generator of the semigroup.
  As an application, we deduce a rate of convergence for Johansson's
  celebrated theorem on linear eigenvalue statistics of Gaussian
  random matrix ensembles.
\end{abstract}

\section{Introduction}
\label{S:intro}

Stein's method is a powerful and general method --- or more
accurately, collection of methods --- introduced by Charles Stein to
bound distances between probability distributions, and thus to prove
probabilistic limit theorems with convergence rates.  Stein's method
in various versions have seen rapid development and many exciting
applications over the last twenty-five years.  We refer to
\cite{Chatterjee-survey} for a brief survey of different approaches to
Stein's method and its history.

The approach favored by Charles Stein himself \cite{Stein-book} is the
\emph{exchangeable pairs} approach, in which one bounds the distance
from a random variable $W$ to a limit distribution by analyzing the
behavior of an exchangeable pair of random variables $(W,W')$.  This
approach is particularly well suited to settings where some kind of
symmetry is explicit, since symmetries can be used to construct
exchangeable pairs.  In \cite{Stein-techreport}, Stein adapted this
approach to a setting in which continuous or ``infinitesimal''
symmetries are present, to prove a version of a theorem of Johansson
\cite{Johansson1} on linear eigenvalue statistics for random unitary
matrices. This ``infinitesimal exchangeable pairs'' proof was later
developed in into a general method by E.\ Meckes, and generalized for
multivariate normal approximation by Chatterjee in E.\ Meckes
\cite{EM-thesis,EM-linear,ChMe,EM-Stein-multi}, with applications in
those papers and \cite{DoSt,Fulman2,FuRo,LaLeWe,EM-eigenfunctions},
among others.

In this paper we show how this infinitesimal exchangeable pairs method
for normal approximation fits naturally with the theory of Markov
diffusion processes.  In particular, when applied to a family of
exchangeable pairs $(X_0,X_t)$ generated by a stationary reversible
Markov process $\{X_t\}$, the hypotheses of the abstract normal
approximation theorems of \cite{EM-linear,ChMe,EM-Stein-multi} are
naturally rephrased in terms of the infinitesimal generator $\Lgen$ of
the process.  We formalize this in Theorem \ref{T:Stein-Markov} below,
which allows one to proceed by analyzing the behavior of $\Lgen$.  We
think of the result as a form of ``the method of exchangeable pairs,
without exchangeable pairs''.

This approach on the one hand facilitates working in settings lacking the
obvious geometric or algebraic symmetries present in many previous
applications, in favor of the more probabilistic symmetry represented
by stationarity of the Markov process.  On the other hand in carrying
out applications one is freed from detailed probabilistic thinking,
and effectively just has to do advanced calculus.

Theorem \ref{T:Stein-Markov} effectively unifies many proofs already
in the literature.  In fact, most if not all of the applications in
the literature of the infinitesimal exchangeable pairs approach can be
realized as applications of Theorem \ref{T:Stein-Markov}, including
the main results of
\cite{DoSt,Fulman2,FuRo,LaLeWe,EM-linear,EM-eigenfunctions}.

As a new application of this method, we present a short and elementary
new proof of a seminal result of Johansson on linear eigenvalue
statistics of Hermitian random matrices. Definitions of some of the
terms appearing in this statement will be given in Section
\ref{S:notation} below.

\begin{thm}[{cf.\ Johansson \cite{Johansson2}}]
  \label{T:Johansson} 
  Let $A \in M_n^{sa}(\C)$ be a GUE matrix (standard normal distribution
  on the space of Hermitian matrices). For each $k$ let
  \[
    W_k = \tr T_k\bigl(\tfrac{1}{2\sqrt{n}}A\bigr) - \E \tr
    T_k\bigl(\tfrac{1}{2\sqrt{n}}A\bigr),
  \]
  where $T_k$ are the Chebyshev polynomials of the first kind.
  
  Then 
  \[
    d_W \left((W_1, \dots, W_d), \bigl(\tfrac{1}{2}Z_1, \tfrac{\sqrt{2}}{2}
      Z_2, \dots, \tfrac{\sqrt{d}}{2} Z_d\bigr)
    \right) \le \frac{K_d}{n},
  \]
  where $Z_1, \dots, Z_d$ are independent standard normal random
  variables and $K_d > 0$ is a constant depending only on $K_d$, and
  $d_W$ denotes the $L^1$-Wasserstein distance.

  The same conclusion holds for the Gaussian Orthogonal Ensemble or
  Gaussian Symplectic Ensemble.
\end{thm}

Our proof of Theorem \ref{T:Johansson} has three major advantages over
Johansson's original proof in \cite{Johansson2}.  Firstly, as is
typical for proofs via Stein's method, we automatically obtain a rate
of convergence, which was not present in \cite{Johansson2}; in this
case the $1/n$ rate is optimal.  Second, our proof is significantly
shorter and more elementary than Johansson's.  Finally, our proof
makes transparent precisely why Chebyshev polynomials of the first
kind are the correct test functions to diagonalize the covariance of
the limiting distribution.

Our methods are extremely flexible and can be applied to study linear
eigenvalue statistics for a variety of different random matrix
ensembles. Applications to several other random matrix ensembles, due
to the first-named author, appear in \cite{Grzybowski-thesis}, and
further applications will appear in forthcoming work.

The recent paper \cite{AnHeMaPo} (written after the results of this
paper were announced) introduced a closely related approach, combining
Markov diffusion semigroups with the approach to Stein's method via
the Malliavin calculus.  The authors of \cite{AnHeMaPo} apply their
techniques to investigate fluctuations of smooth linear statistics of
$\beta$-ensembles, and their results essentially include Theorem
\ref{T:Johansson} with the same rate of convergence.

Our exchange pairs approach and the Malliavin calculus approach of
\cite{AnHeMaPo} each have different advantages.  Although the proof in
\cite{AnHeMaPo} is also significantly shorter than Johansson's
original proof, it builds to some extent on Johansson's approach and
is much more technical than our proof of Theorem \ref{T:Johansson}.
Moreover, we believe that our approach would make it easier to
discover the basic form of results like Theorem \ref{T:Johansson} when
they are not already known (as exemplified by how our proof of Theorem
\ref{T:Johansson} ``discovers'' the Chebyshev polynomials, as seen in
Section \ref{S:GUE} below).

On the other hand, the Malliavin calculus approach utilized in
\cite{AnHeMaPo} appears to be better suited than the exchangeable
pairs approach for handling both more general linear statistics (i.e.,
replacing the polynomials in Theorem \ref{T:Johansson} with smooth
functions) and more general probability metrics (such as
$L^p$-Wasserstein metrics for $p > 1$).  In principle, it is possible
to extend Theorem \ref{T:Johansson} to more general test functions $f$
by approximating those test functions by linear combinations of
Chebyshev polynomials.  Such an approach was taken in the context of
random unitary matrices in \cite{DiEv,DoSt2}.  However, this would
require more precisely controlling the dependence of the constants
$K_d$ on $d$, and result in a loss in the rate of convergence, and we
do not pursue such a generalization in this paper.  Similarly, weaker
(nonoptimal) convergence rates in $L^p$-Wasserstein distance for $p>1$
can be deduced easily from convergence rates for $L^1$-Wasserstein
distance and measure concentration arguments.

Nevertheless, although we do not carry out the details here, in the
univariate case our methods can be used to prove the optimal
convergence rate in total variation distance for traces of individual
Chebyshev polynomials.

\subsection{Definitions and notation}
\label{S:notation}

We write $T_p$ and $U_p$ for $p \ge 0$ for the Chebyshev polynomials
of the first and second kind, respectively, with the index $p$
denoting the degree.  We recall that they are related by the
derivative formula $T_p' = p U_{p-1}$, and that the Chebyshev polynomials
of the second kind satisfy the orthogonality relation
\begin{equation} \label{E:U-orthogonality}
\int_{-1}^1 U_p(x) U_q(x) \sqrt{1-x^2} \ dx = \frac{\pi}{2} \delta_{p,q}.
\end{equation}
We denote by $\sigma$ the semicircle distribution, normalized to have
variance $1$:
\[
  d \sigma(x) = \frac{1}{2\pi} \ind{[-2, 2]}(x) \sqrt{4-x^2} \ dx,
\]
so that \eqref{E:U-orthogonality} can be restated as
\begin{equation} \label{E:semicircle-orthogonality}
  \int U_p(x/2) U_q(x/2) \ d\sigma(x) = \delta_{p,q}.
\end{equation}
These relations, and the fact that $T_p$ is a polynomial of degree
$p$, are the only facts about Chebyshev polynomials that will be
needed in our proof of Theorem \ref{T:Johansson}.

We write $\Mat_{n}(\F)$ for the space of $m \times n$ matrices over
the field $\F$ (where $\F = \R$ or $\C$).  We denote by
$\Mat_n^{sa}(\F)$ the subspace of Hermitian matrices.  Each of these
spaces is equipped with both the operator norm $\norm{\cdot}_{op}$ and
Hilbert--Schmidt norm $\norm{\cdot}_{HS}$, and the Hilbert--Schmidt
inner product $\inprod{A}{B}_{HS} = \tr (AB^*)$.  The standard
Euclidean norm on $\R^n$ is denoted $\abs{\cdot}$.

The $L^1$ Wasserstein distance between two probability measures $\mu$
and $\nu$ on $\R^d$ may be equivalently defined as
\[
d_W(\mu, \nu) = \inf \E \abs{X-Y} = \sup_f \abs{\int f \ d\mu - \int f
\ d\nu}
\]
where the infimum is over pairs of random vectors $X,Y \in \R^d$ with
$X \sim \mu$ and $Y \sim \nu$, and the supremum is over $1$-Lipschitz
functions $f: \R^d \to \R$.  As is usual, we abuse notation and write
$d_W(X,Y)$ for the Wasserstein distance between the distributions of
random vectors $X,Y \in \R^d$.

On any Euclidean space  (that is, a finite-dimensional real inner
product space), we denote by $\nabla$ the gradient operator and
$\Delta$ the Laplacian.

\subsection{Layout of this paper}

In Section \ref{S:Stein-Markov} we will present Theorem
\ref{T:Stein-Markov}, the abstract normal approximation theorem for
functions of Markov diffusion processes that is used to prove our main
results.  In Section \ref{S:GUE} we present in detail the proof of
Theorem \ref{T:Johansson}, by applying Theorem
\ref{T:Stein-Markov} to the Ornstein--Uhlenbeck process in the space
of Hermitian matrices.

Versions of these result and proofs previously appeared in the
first-named author's PhD thesis \cite{Grzybowski-thesis}.

\section{Stein's method and Markov processes}
\label{S:Stein-Markov}

We refer to \cite{BaGeLe} for background and definitions on Markov
processes and semigroups.  We will consider a stationary reversible
Markov process $\{X_t\}_{t \ge 0}$ on a space $S$ with invariant
measure $\mu$.  We denote by $\Lgen$ the infinitesimal generator for
the associated Markov semigroup, defined by
\[
  \Lgen f(x) = \lim_{t\to 0^+} \frac{\E [f(X_t) \mid X_0 = x] - f(x)}{t}
\]
for all $f$ for which this limit exists uniformly, and by
\[
  \Gamma (f,g) = \frac{1}{2} \bigl[ \Lgen(fg) - f\Lgen g - g \Lgen f\bigr]
\]
the associated carr\'e du champ operator. For vector-valued
$F,G: S \to \R^d$ we extend these to vector- and matrix-valued
operators
via
\[
  (\Lgen F)_j = \Lgen (F_j)
  \qquad \text{and} \qquad
  \Gamma(F,G)_{jk} = \Gamma(F_j, G_k).
\]

We emphasize that the role of the generator $\Lgen$ in this work is
completely different from in Barbour's \emph{generator approach} to
Stein's method \cite{Barbour1,Barbour2}.  Barbour's approach is built
around the generator of a process whose stationary distribution is the
limiting distribution (the Gaussian distribution $N(0,\Sigma)$ in our
setting), whereas here $\Lgen$ generates the distribution for an
underlying random object $X_0$, from which the random variables of
interest are constructed.

\begin{lemma}
  \label{Lem:Markov-conditionals}
  Let $\{X_t\}_{t\ge 0}$ be an $S$-valued Markov process with
  infinitesimal generator $\Lgen$ and carr\'e du champ operator
  $\Gamma$, and let $f, g:S \to \R$ be in the domain of $\Lgen$. Then
  \[
    \lim_{t \to 0^+} \frac{1}{t}\E \left[ f(X_t) - f(X_0)\middle\vert X_0\right]
    = \Lgen f(X_0)
  \]
  and
  \[
    \lim_{t \to 0^+} \frac{1}{t}\E \left[ \bigl(f(X_t) - f(X_0)\bigr)
      \bigl(g(X_t) - g(X_0)\bigr) \middle\vert X_0\right] = 2\Gamma
    (f, g) (X_0).
  \]
\end{lemma}

\begin{proof}
  The first claim follows immediately from the definition of
  $\Lgen$. The second follows from the definition of $\Gamma$ after
  rewriting
  \begin{multline*}
    \bigl(f(X_t) - f(X_0)\bigr) \bigl(g(X_t) - g(X_0)\bigr) \\
    = \bigl(f(X_t) g(X_t) - f (X_0) g (X_0) \bigr)
    - f(X_0) \bigl(g(X_t) - g(X_0) \bigr)
    - g(X_0) \bigl(f(X_t) - f(X_0) \bigr).
    \qedhere
  \end{multline*}
\end{proof}

Using Lemma \ref{Lem:Markov-conditionals}, we can apply the
multivariate infinitesimal version of Stein's method of exchangeable
pairs in \cite[Theorem 4]{EM-Stein-multi}\footnote{Prof.\ C.\ Houdr\'e
  has pointed out to us that there is an error in the statement and
  proof of Lemma 2 of \cite{EM-Stein-multi}, which is corrected in
  \cite[Remark 4.2]{ArHo}.  This correction results in a slight change
  in a constant that appears in \cite[Theorem 3]{EM-Stein-multi}, but
  does not affect the statement of \cite[Theorem 4]{EM-Stein-multi}
  used here.}  to the family of exchangeable pairs $(X_0, X_t)$.  This
yields the following result, which we can think of as a version of
``Stein's method of exchangeable pairs without exchangeable pairs''.

\begin{thm}
  \label{T:Stein-Markov}
  Let $\{X_t\}_{t\ge 0}$ be a stationary reversible $S$-valued Markov
  process with infinitesimal generator $\Lgen$ and carr\'e du champ
  operator $\Gamma$, and let $F:S \to \R^d$ be in the domain of
  $\Lgen$ satisfying $\E F(X_0) = 0$.  Suppose that there exist
  \begin{itemize}
  \item an invertible deterministic matrix $\Lambda \in M_d(\R)$,
  \item a positive semidefinite deterministic matrix
    $\Sigma \in M_d^{sa}(\R)$,
  \item an $X_0$-measurable random vector $E_1 \in \R^d$, and
  \item an $X_0$-measurable random matrix $E_2 \in M_d(\R)$
  \end{itemize}
  such that:
  \begin{gather}
    \Lgen F(X_0) = - \Lambda F(X_0) + E_1,  \label{E:Lf} \\
    \Gamma(F,F)(X_0) = \Lambda \Sigma + E_2, \label{E:Gamma}
  \end{gather}
  and for each $\rho > 0$,
  \begin{equation}
    \label{E:Lindeberg}
    \lim_{t \to 0^+} \frac{1}{t} \E \left[ \abs{F(X_t) -
        F(X_0)}^2 \ind{\abs{F(X_t) -F( X_0)} > \rho}\right] = 0.
  \end{equation}
  Then for each $g \in C^2(\R^d)$,
  \[
    \abs{\E g(F(X_0)) - \E g(Z_\Sigma)} \le \norm{\Lambda^{-1}}_{op}
    \left[\norm{ \abs{\nabla g} }_\infty \E \abs{E_1} + \frac{1}{2}
      \norm{ \norm{D^2 g}_{HS} }_\infty \E \norm{E_2}_{HS}\right],
  \]
  where $Z_\Sigma \sim N(0,\Sigma)$ and $D^2 g$ denotes the Hessian
  matrix of $g$.

  If $\Sigma$ is nonsingular, then moreover
  \[
    d_W(F(X_0), Z_\Sigma) \le \norm{\Lambda^{-1}}_{op}
    \left[ \E \abs{E_1} + \norm{\Sigma^{-1/2}}_{op} \E \norm{E_2}_{HS}\right].
  \]
\end{thm}

When $d=1$ one can deduce a version of Theorem
\ref{T:Stein-Markov} with a bound on the total variation distance by
using a modification of \cite[Theorem 1]{EM-linear} in place of
\cite[Theorem 4]{EM-Stein-multi}.  (It is necessary to slightly
generalize \cite[Theorem 1]{EM-linear} to accommodate a nonzero error
term $E_1$ in condition \eqref{E:Lf}.)  For the sake of brevity, we
will pursue this in the present paper.

As mentioned in the introduction, most if not all of the applications
in the literature of \cite[Theorem 4]{EM-Stein-multi} and its
immediate precedents in \cite{EM-linear,ChMe} can be realized as
applications of Theorem \ref{T:Stein-Markov}.  We note in particular
that when the process $\{X_t\}_{t\ge 0}$ is Brownian motion on a
compact Riemannian manifold $S$, whose infinitesimal generator is the
Laplace--Beltrami operator $\Delta_S$, then condition \eqref{E:Lf} of
Theorem \ref{T:Stein-Markov} is trivially satisfied with a diagonal
matrix $\Lambda$ and $E_1 = 0$ if the components of $F$ are
eigenfunctions of $\Delta_S$.  This immediately suggests conjecturing
the results on such eigenfunctions proved in
\cite{EM-eigenfunctions,EM-Stein-multi}, and indeed produces more
streamlined proofs of those results.

The Lindeberg-like condition \eqref{E:Lindeberg} in Theorem
\ref{T:Stein-Markov} is typically easy to dispense with in
applications by establishing the stronger condition
\begin{equation}
  \label{E:3rd-moment}
  \lim_{t \to 0^+} \frac{1}{t}\E \abs{F(X_t) - F(X_0)}^3 = 0.
\end{equation}
In our present applications \eqref{E:3rd-moment} follows easily from
measure concentration results; see \cite[Theorem
3.6]{Grzybowski-thesis} for a more general sufficient condition.
Theorem \ref{T:Stein-Markov} thus essentially reduces the normal
approximation of $F(X_0)$ to computation of $\Lgen F$ and
$\Gamma(F,F)$, and estimation of some related expectations.

\section{Linear eigenvalue statistics for the GUE}
\label{S:GUE}

In this section we prove Theorem \ref{T:Johansson} by applying
Theorem \ref{T:Stein-Markov} to an Ornstein--Uhlenbeck
process. For brevity, we will focus on the case of the Gaussian
Unitary Ensemble.  Only minor adjustments are needed in the cases of
the GOE and GSE.

On an arbitrary Euclidean space the Ornstein--Uhlenbeck process
process has generator and carr\'e du champ
\begin{equation}
  \label{E:OU-L-Gamma}
  \Lgen f(x) = \Delta f(x) - \inprod{x}{\nabla f(x)}
  \qquad \text{and} \qquad
  \Gamma(f,g) = \inprod{\nabla f}{\nabla g}
\end{equation}
and the stationary distribution is the standard Gaussian distribution.
We will consider this process on the Euclidean space $\Mat_n^{sa}(\C)$
equipped with the inner product $\inprod{A}{B}_{HS}$, so that $X_0$
has the distribution of the Gaussian Unitary Ensemble.

We will first need the following elementary derivative
computations. Here are below we will adopt the notations
\[
  f^{\tr}(A) = \tr f(A)
  \qquad \text{and} \qquad
  \widetilde{f}^{\tr}(A) = \tr f (n^{-1/2}A)
\]
for $f: \R \to \R$ and $A \in \Mat_n^{sa}(\C)$.

\begin{lemma}
  \label{Lem:GUE-derivs}
  Given a polynomial function $f:\R \to \R$, 
  $\nabla f^{\tr}(A) = f'(A)$ and $\nabla \widetilde{f}^{\tr}(A) =
  n^{-1/2} f'(n^{-1/2}A)$.

  If $g_p(x) = x^p$, then
  \[
    \Delta g_p^{\tr}(A) = p \sum_{k=0}^{p-2} \tr(A^k) \tr(A^{p-2-k})
  \]
  and
  \[
    \Delta \widetilde{g}_p^{\tr}(A) = \frac{p}{n} \sum_{k=0}^{p-2}
    \widetilde{g}_k^{\tr}(A) \widetilde{g}_{p-2-k}^{\tr}(A).
  \]
\end{lemma}

\begin{proof}
  For $H \to 0$ we have the expansion
  \begin{align*}
    \tr(A+H)^p - \tr A^p
    & = \tr \sum_{j=0}^{p-1} A^j H A^{p-1-j}
      + \tr \sum_{j=0}^{p-2} \sum_{k=0}^{p-2-j} A^j H A^k H A^{p-2-j-k} + O(\norm{H}_{HS}^3).
  \end{align*}
  The first gradient expression immediately follows since
  \[
    \tr \sum_{j=0}^{p-1} A^j H A^{p-1-j} = \inprod{p A^{p-1}}{H}
  \]
  by the cyclic property of the trace, and the second gradient
  expression follows by the chain rule.

  Now let $\mathcal{B}$ be the orthonormal basis
  \[
    \mathcal{B} = \Set{E_{jj}}{1 \le j \le n} \cup
    \Set{\frac{E_{jk}+E_{kj}}{\sqrt{2}}}{1 \le j < k \le n} \cup
    \Set{\frac{i(E_{jk}-E_{kj})}{\sqrt{2}}}{1 \le j < k \le n}
  \]
  of $\Mat_n^{sa}(\C)$, where $E_{jk} = e_j e_k^*$. It is straightforward
  to check that
  \[
    \sum_{H \in \mathcal{B}} HBH = \tr(B) I_n
  \]
  for any $B \in \Mat_n^{sa}(\C)$. Combining this with the
  second-order terms in the asymptotic expansion above yields
  \begin{align*}
    \Delta g_p^{\tr} & = 2\sum_{j=0}^{p-2} \sum_{k=0}^{p-2-j} \tr(A^k)
    \tr(A^{p-2-k}) = 2\sum_{k=0}^{p-2} (p-1-k) \tr(A^k) \tr(A^{p-2-k}) \\
    & = p \sum_{k=0}^{p-2} \tr(A^k) \tr(A^{p-2-k}).
  \end{align*}
  Here the second equality follows by reversing the order of the sums,
  and the third follows by averaging the previous expression with the
  corresponding version obtained by replacing $k$ by $p-2-k$.

  The second Laplacian expression follows from the first by the chain rule.
\end{proof}

The above computation of the Laplacian of a linear eigenvalue
statistic is the only part of the proof that needs to be adjusted for
the GOE and GSE; see \cite{Grzybowski-thesis} for details.

Our proof of Theorem \ref{T:Johansson} will require a rate of
convergence for expectations of linear eigenvalue statistics.  This
convergence rate is essentially known, but in order to illustrate how
self-contained our approach is we will give a new, easy
Markov process-based proof.

\begin{lemma}
  \label{Lem:ss-rate}
  Let $A$ have the distribution of the GUE.  Then for each integer
  $p \ge 0$, 
  \[
    \E \tr \bigl(\tfrac{1}{\sqrt{n}} A\bigr)^{2p} = nC_p + O(1)
  \]
  for $n \to \infty$, where $C_p = \frac{1}{p+1} \binom{2p}{p}$ is the
  $p^\mathrm{th}$ Catalan number, and
  $\E \tr \bigl(\tfrac{1}{\sqrt{n}} A\bigr)^{2p+ 1} = 0$.

  Consequently, 
  \[
    \abs{\E \widetilde{f}^{\tr} ( A ) - n \int f \
      d\sigma} \le K_f
  \]
  for any polynomial $f$.
\end{lemma}

\begin{proof}
   Write $g_p(x) = x^p$.  Lemma \ref{Lem:GUE-derivs} implies that for
  $p \ge 2$, 
  \begin{equation} \label{E:GUE-L}
    \Lgen g_p^{\tr}(A)  = \Delta g^{\tr}_p(A) - \inprod{p A^{p-1}}{A}_{HS} 
    = p \sum_{k=0}^{p-2} g^{\tr}_k(A) g^{\tr}_{p-2-k}(A) - p g^{\tr}_p(A)
  \end{equation}

  It follows from the definition of the infinitesimal generator
  $\Lgen$ that $\E \Lgen f (X_0)= 0$ for any stationary Markov process
  and any $f$ in the domain of $\Lgen$.  Applied to \eqref{E:GUE-L},
  this implies that
  \[
    \E \widetilde{g}^{\tr}_p(A) = \frac{1}{n} \sum_{k=0}^{p-2} \E \widetilde{g}^{\tr}_k(A)
    \widetilde{g}_{p-2-k}^{\tr}(A).
  \]
  Corollary 2 of \cite{MeSz} implies that
  $\var \widetilde{g}^{\tr}_p(A) = O(1)$, so that
  \[
    \abs{\Cov\bigl(\widetilde{g}^{\tr}_k(A), \widetilde{g}^{\tr}_{p-2-k}(A)\bigr)}
    \le \sqrt{\var \widetilde{g}^{\tr}_p(A) \var \widetilde{g}^{\tr}_{p-2-k}(A)} = O(1),
  \]
  and thus
  \begin{equation} \label{E:GUE-expectation-recursion}
    \E \widetilde{g}^{\tr}_p(A) = \frac{1}{n}\sum_{k=0}^{p-2} \E \widetilde{g}^{\tr}_k(A) \E
    \widetilde{g}^{\tr}_{p-2-k}(A) + O(1).
  \end{equation}

  By the symmetry of the distribution of $A$, we have $\E \widetilde{g}^{\tr}_p(
  A) = 0$ whenever $p$ is odd.  If we write $M_{n,p} = n^{-1} \E
  \widetilde{g}^{\tr}_{2p}(A)$, then \eqref{E:GUE-expectation-recursion}
  simplifies to
  \begin{equation*} %\label{E:GUE-even-expectation-recursion}
    M_{n,p} = \sum_{k=0}^{p-1} M_{n,k} M_{n,p-1-k} + O(n^{-1})
  \end{equation*}
  for $p \ge 1$.  Since $M_{n,0} = C_0 = 1$ and the Catalan numbers satisfy the
  recursion relation
  \begin{equation}\label{E:Catalan-recursion}
    C_p = \sum_{k=0}^{p-1} C_k C_{p-1-k},
  \end{equation}
  it follows by induction that $M_{n,p} = C_p + O(n^{-1})$, which
  completes the proof of the first statement.

  The second statement then follows from the fact that
  \[
    \int x^k \ d\sigma(x) = \begin{cases} C_{k/2} & \text{ if $k$ is
        even}, \\ 0 & \text{ if $k$ is odd.} \end{cases}
    \qedhere
  \]
\end{proof}

As discussed in the introduction, one advantage of our approach to
Theorem \ref{T:Johansson} via Theorem \ref{T:Stein-Markov} is that
it makes clear how the Chebyshev polynomials arise naturally.  Before
beginning the proof of Theorem \ref{T:Johansson} we will give the
intuition behind this part of the argument.

In the setting of Theorem \ref{T:Stein-Markov}, recall that the
generator $\Lgen$ is symmetric on $L^2(\mu)$, where $\mu$ is the
stationary distribution (i.e., the distribution of
$X_0$). Eigenfunctions of $\Lgen$ corresponding to distinct
eigenvalues are then orthogonal in $L^2(\mu)$.  Since
also
\[
  \E \Gamma(f,g) (X_0) = - \E \bigl[f(X_0) \Lgen g(X_0)\bigr]
\]
by symmetry, it follows that if $f$ and $g$ are eigenfunctions of
$\Lgen$ corresponding to distinct eigenvalues, then $f$ and $g$ are
also orthogonal with respect to the bilinear form $\E \Gamma$.  Thus
the operator $\Lgen$ and the bilinear form $\E \Gamma$ are
simultaneously diagonalized by a basis consisting of such
eigenfunctions.  So if we choose the components of $F$ to be
eigenfunctions of $\Lgen$ with distinct eigenvalues, then the matrix
$\Lambda$ will be diagonal, and our natural candidate for
$\Lambda \Sigma$ is also diagonal, so that $\Sigma$ is diagonal as
well.

In the setting of Theorem \ref{T:Johansson} the eigenfunctions of
$\Lgen$ are not linear eigenvalue statistics in general.  However, it
transpires that linear eigenvalue statistics can be found that are
good approximations to eigenfunctions of $\Lgen$, and that these
approximate eigenfunctions are more easily identified by analyzing
$\Gamma$ rather than $\Lgen$.  Combining Lemma \ref{Lem:GUE-derivs}
with the fact that the semicircle law $\sigma$ is the limiting
spectral distribution (as shown here in Lemma \ref{Lem:ss-rate}) shows
that if $f$ and $g$ are polynomials, then
\[
  \E \Gamma(\widetilde{f}^{\tr},\widetilde{g}^{tr})(X_0) \approx \int f'
  g' \ d\sigma.
\]
Our natural candidates to diagonalize both $\Lambda$ and $\Sigma$ are
therefore polynomials whose derivatives are orthogonal in
$L^2(\sigma)$.  The orthogonal polynomials for $\sigma$ are
$U_p(x/2)$, leading us to their antiderivatives (up to normalization),
$T_p(x/2)$.

\begin{proof}[Proof of Theorem \ref{T:Johansson}]
  Implementing the strategy described above, we begin by considering
  the carr\'e du champ operator for the Ornstein--Uhlenbeck process on
  $\Mat_n^{sa}(\C)$. Let $f$ and $g$ be polynomials.  By
  \eqref{E:OU-L-Gamma} and Lemma \ref{Lem:GUE-derivs} we have
  \begin{equation*}
    \label{E:GUE-Gamma}
    \Gamma(\widetilde{f}^{\tr},\widetilde{g}^{tr})(X) = \frac{1}{n}
    \tr \bigl(f'(n^{-1/2}X) g'(n^{-1/2}X)\bigr),
  \end{equation*}
  where we write $X$ instead of $X_0$ to simplify the notation.
  Lemma \ref{Lem:ss-rate} then implies that
  \begin{equation*}
    \E \Gamma(\widetilde{f}^{\tr},\widetilde{g}^{tr})(X) = \int f' g' \ d\sigma
    + O(n^{-1}),
  \end{equation*}
  while \cite[Corollary 2]{MeSz} implies that
  \[
    \E \abs{\Gamma(\widetilde{f}^{\tr},\widetilde{g}^{tr})(X) - \E
      \Gamma(\widetilde{f}^{\tr},\widetilde{g}^{tr})(X) } = O(n^{-1}),
  \]
  and so
  \begin{equation}
    \label{E:GUE-E-Gamma}
    \E \abs{\Gamma(\widetilde{f}^{\tr},\widetilde{g}^{tr})(X) -
      \int f' g' \ d\sigma} = O(n^{-1}).
  \end{equation}
  
  If we now define
  \[
    f_p(x) = T_p(x/2) - \E \tr T_p\left(\frac{1}{2\sqrt{n}} X\right)
  \]
  for $p \ge 1$, then by \eqref{E:GUE-E-Gamma} and
  \eqref{E:semicircle-orthogonality},
  \[
    \E \abs{ \Gamma(\widetilde{f}_p^{\tr}, \widetilde{f}_q^{\tr})(X)
      - \frac{p^2}{4} \delta_{p,q}} = O(n^{-1}).
  \]
  We therefore define $F$ with components $F_p = \widetilde{f}_p^{\tr}$
  for $1 \le p \le d$, and it follows that
  \begin{equation}\label{E:Gamma-GUE}
    \Gamma (F,F)(X) = \frac{1}{4} \diag(1^2, 2^2, \dots, d^2) + E_2,
  \end{equation}
  where $E_2 \in M_d(\R)$ is a random matrix with $\E \norm{E_2}_{HS}
  = O(n^{-1})$.  (Here and below we leave the dependence of $O$
  expressions on $d$ implicit.)

  We now turn to the analysis of the infinitesimal generator $\Lgen$.
  Again denote $g_p(x) = x^p$.  By \eqref{E:OU-L-Gamma} and Lemma
  \ref{Lem:GUE-derivs},
  \begin{equation} \label{E:L-powers}
    \begin{split}
      \Lgen \widetilde{g}_p^{\tr} (X)
      & = \frac{p}{n} \sum_{k=0}^{p-2}
      \widetilde{g}_k^{\tr}(X) \widetilde{g}_{p-2-k}^{\tr}(X) - n^{-1/2}
      \inprod{X}{g'_p(X)} \\
      & = - p \widetilde{g}_p^{\tr}(X) - (2 - \delta_{p,2}) p
      \widetilde{g}^{\tr}_{p-2}(X) + \frac{p}{n} \sum_{k=1}^{p-3}
      \widetilde{g}_k^{\tr}(X) \widetilde{g}_{p-2-k}^{\tr}(X).
    \end{split}
  \end{equation}
  (Here we consider a sum with no valid summands to be $0$, and also
  set $\widetilde{g}^{\tr}_k = 0$ if $k < 0$.) 

  Denote
  $G_p(X) = \widetilde{g}_p^{\tr}(X) - \E \widetilde{g}_p^{\tr}(X)$.
  By Lemma \ref{Lem:ss-rate}, \eqref{E:Catalan-recursion}, and
  \cite[Corollary 2]{MeSz}, \eqref{E:L-powers} reduces to 
  \begin{equation} \label{E:L-centered-powers}
    \Lgen G_p(X) = -p G_p + 2p \sum_{k=0}^{\lfloor p/2 \rfloor - 1}
    C_k G_{p-2(k+1)} + O(n^{-1}).
  \end{equation}
  Here the $O(n^{-1})$ term stands for a random variable whose $L^1$
  norm is $O(n^{-1})$.

  Since the Chebyshev polynomials up to degree $d$ are a basis of the
  space of polynomials of degree at most $d$, it follows that
  $F_1, \dots, F_d$ span the space of mean-zero polynomial linear
  eigenvalue statistics of $X$ of degree at most $d$.  In particular,
  $F_p$ is a linear combination of $G_1, \dots, G_p$, and also
  $G_k$ is a linear combination of $F_1, \dots, F_k$.  From this and
  \eqref{E:L-centered-powers} it follows that
  \begin{equation} \label{E:L-F}
    \Lgen F_p(X) = -p F_p(X) - \sum_{k=1}^{p-1} \Lambda_{p,k} F_k(X) +
    O(n^{-1}) 
  \end{equation}
  for some constants $\Lambda_{p,k}$.  If we complete these constants
  to a matrix $\Lambda \in M_d(\R)$ by setting $\Lambda_{pp} = p$ and
  $\Lambda_{pk} = 0$ for $k > p$, then condition \eqref{E:Lf} holds
  for a random vector $E_1$ with $\E \abs{E_1} = O(n^{-1})$.

  We claim that $\Lambda$ is in fact diagonal.  Since $\Lambda$ is
  lower triangular, its eigenvalues are its diagonal entries
  $1, 2, \dots, d$.  For each $1\le p \le d$, let
  $v^p = (v^p_1,\dots,v^p_d)$ be an eigenvector of $\Lambda^T$ with
  eigenvalue $p$, and define $H_p = \sum_{i=1}^d v^p_i F_i$.  Then
  \begin{equation*} \begin{split}
      LH_p(X) & = \sum_{i=1}^d
      v^p_i \left(-\sum_{j=1}^d \Lambda_{ij} F_j(X) + E_{1,i}\right)
      = - \sum_{j=1}^d F_j(X) (\Lambda^T v^p)_j + \inprod{v}{E_1} \\
      & = - p \sum_{j=1}^d v^p_j F_j(X) + \inprod{v}{E_1} = -p H_p(X) +
      O(n^{-1}).
    \end{split}
  \end{equation*}
  From this it follows that
  \[
    \E \Gamma(H_p, H_q)(X) = - \E \bigl[H_p(X) \Lgen H_q(X)] = p \E \bigl[H_p(X)H_q(X)\bigr] +
    O(n^{-1}) 
  \]
  and similarly
  \[
    \E \Gamma(H_p, H_q)(X) = q \E \bigl[ H_p(X)H_q(X) \bigr] + O(n^{-1}),
  \]
  so if $p\neq q$ then $\E
  \Gamma(H_p, H_q)(X) = O(n^{-1})$.

  On the other hand, \eqref{E:Gamma-GUE} implies that
  \[
    \Gamma(H_p, H_q)(X) = \frac{1}{4} \sum_{k=1}^d k^2 v_k^p v_k^q +
    O(n^{-1}).
  \]
  Combining these observations, we have that
  \[
    \sum_{k=1}^d k^2 v_k^p v_k^q = 0
  \]
  whenever $p \neq q$.  Note that $v^1$ is a scalar multiple of
  the first standard basis vector.  It now follows by an easy
  induction argument that each $v^p$ is a scalar multiple of the
  $p^\mathrm{th}$ standard basis vector, and thus $\Lambda$ is
  diagonal.  More precisely, $\Lambda = \diag(1, 2, \dots, p)$.

  We have thus shown that conditions \eqref{E:Lf} and \eqref{E:Gamma}
  in Theorem \ref{T:Stein-Markov} hold with
  \[
    \Lambda = \diag(1, 2, \dots, p), \qquad
    \Sigma = \frac{1}{4} \diag(1, 2, \dots, p),
  \]
  and $E_1, E_2$ satisfying $\E \abs{E_1} = O(n^{-1})$ and
  $\E \norm{E_2}_{HS} = O(n^{-1})$.
    
  To deduce Theorem \ref{T:Johansson} it remains only to verify
  \eqref{E:Lindeberg}, which we will do using the third moment
  condition \eqref{E:3rd-moment}.  For this we note that for the
  Ornstein--Uhlenbeck process, conditional on $X_0$, $X_t$ is
  distributed as $e^{-t} X_0 + \sqrt{1-e^{-2t}} Z$, where $Z$ is a
  standard Gaussian random vector (GUE in our present setting).  Using
  this representation, \cite[Corollary 2]{MeSz} implies that
  \[
    \E \bigl[\abs{F_p(X_t) - F_p(X_0) }^3\big\vert X_0\bigr] \le C
    (1-e^{-2t})^{3/2}
  \]
  for each $p$, which implies that \eqref{E:3rd-moment} holds, and
  completes the proof.  
\end{proof}

\bibliographystyle{plain}

\bibliography{johansson-roc}

\begin{thebibliography}{10}

\bibitem{AnHeMaPo}
J.~Angst, R.~Herry, D.~Malicet, and G.~Poly.
\newblock Sharp total variation rates of convergence for fluctuations of linear
  statistics of $\beta$-ensembles.
\newblock {\em Ann. Henri Poincar\'e}.
\newblock To appear.

\bibitem{ArHo}
B.~Arras and C.~Houdr\'{e}.
\newblock Covariance representations, {$L^p$}-{P}oincar\'{e} inequalities,
  {S}tein's kernels, and high-dimensional {CLT}s.
\newblock In {\em High Dimensional Probability {IX}---The Ethereal Volume},
  volume~80 of {\em Progr. Probab.}, pages 3--73. Birkh\"{a}user/Springer,
  Cham, 2023.

\bibitem{BaGeLe}
D.~Bakry, I.~Gentil, and M.~Ledoux.
\newblock {\em Analysis and geometry of {M}arkov diffusion operators}, volume
  348 of {\em Grundlehren der mathematischen Wissenschaften}.
\newblock Springer, Cham, 2014.

\bibitem{Barbour1}
A.~D. Barbour.
\newblock Stein's method and {P}oisson process convergence.
\newblock In {\em A celebration of applied probability. J. Appl. Probab.},
  pages 175--184. 1988.

\bibitem{Barbour2}
A.~D. Barbour.
\newblock Stein's method for diffusion approximations.
\newblock {\em Probab. Theory Related Fields}, 84(3):297--322, 1990.

\bibitem{Chatterjee-survey}
S.~Chatterjee.
\newblock A short survey of {S}tein's method.
\newblock In {\em Proceedings of the {I}nternational {C}ongress of
  {M}athematicians---{S}eoul 2014. {V}ol. {IV}}, pages 1--24. Kyung Moon Sa,
  Seoul, 2014.

\bibitem{ChMe}
S.~Chatterjee and E.~Meckes.
\newblock Multivariate normal approximation using exchangeable pairs.
\newblock {\em ALEA Lat. Am. J. Probab. Math. Stat.}, 4:257--283, 2008.

\bibitem{DiEv}
P.~Diaconis and S.~N. Evans.
\newblock Linear functionals of eigenvalues of random matrices.
\newblock {\em Trans. Amer. Math. Soc.}, 353(7):2615--2633, 2001.

\bibitem{DoSt}
C.~D\"obler and M.~Stolz.
\newblock Stein's method and the multivariate {CLT} for traces of powers on the
  classical compact groups.
\newblock {\em Electron. J. Probab.}, 16:no. 86, 2375--2405, 2011.

\bibitem{DoSt2}
C.~D\"{o}bler and M.~Stolz.
\newblock A quantitative central limit theorem for linear statistics of random
  matrix eigenvalues.
\newblock {\em J. Theoret. Probab.}, 27(3):945--953, 2014.

\bibitem{Fulman2}
J.~Fulman.
\newblock Stein's method, heat kernel, and traces of powers of elements of
  compact {L}ie groups.
\newblock {\em Electron. J. Probab.}, 17:no. 66, 16, 2012.

\bibitem{FuRo}
J.~Fulman and A.~R\"{o}llin.
\newblock Stein's method, heat kernel, and linear functions on the orthogonal
  groups.
\newblock {\em J. Algebra}, 607:272--285, 2022.

\bibitem{Grzybowski-thesis}
D.~Grzybowski.
\newblock {\em Limit Theorems via Concentration of Measure}.
\newblock PhD thesis, Case Western Reserve University, 2023.

\bibitem{Johansson1}
K.~Johansson.
\newblock On random matrices from the compact classical groups.
\newblock {\em Ann. of Math. (2)}, 145(3):519--545, 1997.

\bibitem{Johansson2}
K.~Johansson.
\newblock On fluctuations of eigenvalues of random {H}ermitian matrices.
\newblock {\em Duke Math. J.}, 91(1):151--204, 1998.

\bibitem{LaLeWe}
G.~Lambert, M.~Ledoux, and C.~Webb.
\newblock Quantitative normal approximation of linear statistics of
  {$\beta$}-ensembles.
\newblock {\em Ann. Probab.}, 47(5):2619--2685, 2019.

\bibitem{EM-thesis}
E.~Meckes.
\newblock {\em An infinitesimal version of {S}tein's method of exchangeable
  pairs}.
\newblock ProQuest LLC, Ann Arbor, MI, 2006.
\newblock Thesis (Ph.D.)--Stanford University.

\bibitem{EM-linear}
E.~Meckes.
\newblock Linear functions on the classical matrix groups.
\newblock {\em Trans. Amer. Math. Soc.}, 360(10):5355--5366, 2008.

\bibitem{EM-Stein-multi}
E.~Meckes.
\newblock On {S}tein's method for multivariate normal approximation.
\newblock In {\em High Dimensional Probability {V}: the {L}uminy Volume},
  volume~5 of {\em Inst. Math. Stat. (IMS) Collect.}, pages 153--178. Inst.
  Math. Statist., Beachwood, OH, 2009.

\bibitem{EM-eigenfunctions}
E.~Meckes.
\newblock On the approximate normality of eigenfunctions of the {L}aplacian.
\newblock {\em Trans. Amer. Math. Soc.}, 361(10):5377--5399, 2009.

\bibitem{MeSz}
M.~Meckes and S.~Szarek.
\newblock Concentration for noncommutative polynomials in random matrices.
\newblock {\em Proc. Amer. Math. Soc.}, 140(5):1803--1813, 2012.

\bibitem{Stein-book}
C.~Stein.
\newblock {\em Approximate Computation of Expectations}, volume~7 of {\em
  Institute of Mathematical Statistics Lecture Notes---Monograph Series}.
\newblock Institute of Mathematical Statistics, Hayward, CA, 1986.

\bibitem{Stein-techreport}
C.~Stein.
\newblock The accuracy of the normal approximation to the distribution of the
  traces of powers of random orthogonal matrices.
\newblock Technical Report 470, Stanford University Dept. of Statistics, 1995.

\end{thebibliography}

\end{document}